\documentclass[12pt,oneside]{amsart}
\usepackage{geometry}                
\usepackage{graphicx}
\usepackage{enumerate}
\usepackage{amsthm,amssymb,amsmath}
\usepackage{mathtools}
\allowdisplaybreaks[3]
\title[{\small{\sc J. Balogh, Z.~F\"uredi, S.~Roy:   Sidon sets, \enskip
}}]{An upper bound on the size of Sidon sets}

\newtheorem{theorem}{Theorem}[section]
\newtheorem{lemma}[theorem]{Lemma}

\newtheorem{claim}[theorem]{Claim}
\theoremstyle{definition}

\newcommand{\Claim}[2]{\noindent {\bf Claim #1:} \emph{#2}}

\numberwithin{equation}{section}

\def\ZZ{\mathbb{Z}}
\def\NN{\mathbb{N}}
\def\FF{\mathbb{F}}

\def\A{\mathcal{A}}

\def\cP{\mathcal{P}}

\def\cA{{\mathcal{A}}}
\def\cE{{\mathcal{E}}}

\def\cL{{\mathcal{L}}}
\def\eps{\varepsilon}
\def\lll{t}   
\def\MMM{\ell}   
\def\mmm{m}  

\begin{document}

\author{J\'ozsef Balogh}
 \thanks{Department of Mathematical Sciences, University of Illinois at Urbana-Champaign, IL, USA. e-mail: \texttt{jobal@illinois.edu}. Research supported by NSF RTG Grant DMS-1937241, NSF Grant DMS-1764123 and Arnold O. Beckman Research Award (UIUC) Campus Research Board 18132, the Langan Scholar Fund (UIUC)  and the Simons Fellowship.}
 \author{Zolt\'an F\"uredi}
 \thanks{Alfr\' ed R\' enyi Institute of Mathematics, Budapest, Hungary.
   e-mail: \texttt{z-furedi@illinois.edu.} Research was supported in part by NKFIH grant KH130371 and NKFI–133819.}
\author{Souktik Roy}
  \thanks{Department of Mathematics, University of Illinois at Urbana-Champaign, Urbana, IL, USA. e-mail: \texttt{souktik2@illinois.edu}.}

\begin{abstract}
In this entry point into the subject, combining two elementary proofs, we decrease the gap between the upper
and lower bounds  by $0.2\%$ in
a classical combinatorial number theory problem. We show that the
maximum size of a Sidon set of $\{ 1, 2, \ldots, n\}$ is at most
$\sqrt{n}+ 0.998n^{1/4}$ for sufficiently large $n$.
\end{abstract}

\maketitle

\section{History}

In 1932 S.~Sidon asked a question of a fellow student P. Erd{\H{o}}s.
Their advisor was L.~Fej\'er, an outstanding mathematician (cf.~Fej\'er kernel) working on summability of infinite series, who  had a number of outstanding students who contributed to mathematical analysis
(M.~Fekete 1909 [Fekete's Lemma, see~\cite{FR}], G.~P\'olya 1912, John von Neumann 1926, P.~Erd\H os 1934, P. Tur\'an 1935, V.~T.~S\'os 1957). Sidon~\cite{S32} studying the $L_p$ norm of certain Fourier series proposed the
 following problem, which we present here in contemporary wording.

A set of numbers $A$ is  a {\em Sidon set}, or alternately a \emph{$B_2$-set}, if $a+b=c+d$, $a,b,c,d\in A$ imply $\{ a,b\}=\{ c,d\}$, i.e., all pairwise sums are distinct.

Let $S(n)$ denote the maximum size of a Sidon set can have when $A\subset \{ 1,2,\dots, n\} $ $= : [n]$. It is obvious that $S(n)\leq 2 \sqrt{n}$ since $\binom{|A|+1}{2}\leq 2n-1$. The sequence of powers of $2$, i.e.,  $1, 2, 4, 8,\dots$ is an (infinite) Sidon set showing $S(n)> \log_2 n$.
It is rather difficult to construct large Sidon sets, but Sidon found one showing $S(n)> n^{1/4}$.
Erd\H os immediately observed that the greedy algorithm gives $S(n)\geq n^{1/3}$ as follows.
If $A\subset [n]$ and $n> |A|^3$ then one can always find an $x\in [n]$ such that $x$ cannot be written as $x=a+b-c$, $a,b,c\in A$. Then $A\cup \{ x\}$ is a Sidon set as well.
In 1941, Erd\H{o}s and Tur\'an~\cite{ET41} observed that a result of Singer~\cite{S38} implies that $S(n)>\sqrt{n}$ infinitely many times. They also proved that
  \begin{equation*}
  S(n) \leq n^{1/2}+O(n^{1/4}).
\end{equation*}
Lindstr\"{o}m~\cite{L69}  showed in 1969,  that for all $n$
 \begin{equation}\label{eq142}
  S(n) < n^{1/2}+n^{1/4}+1.
\end{equation}
This was slightly improved by Cilleruelo~\cite{C10} in 2010
 \begin{equation}\label{eq143}
  S(n) < n^{1/2}+n^{1/4}+\frac{1}{2}.
\end{equation}
The study of Sidon sets 
became
a classic topic of additive number theory; see e.g., the survey by O’Bryant~\cite{B04}. The notion can be
extended in a natural way to (finite) groups and fields of characteristic zero see, e.g., Cilleruelo~\cite{C12}.
Many authors apply deep tools, although Ben Green~\cite{G01} notes that some approaches via Fourier analysis have the same spirit as the one in Erd\H{o}s-Tur\'an~\cite{ET41}.

The upper bound in~\eqref{eq142} has basically remained unmoved since 1969.
Erd\H{o}s~\cite{E94} offered \$500 for  a proof or disproof that for every $\eps>0$ the equality
  $S(n)< \sqrt{n} + o(n^\eps)$ holds.

\subsection{The aim of this paper, the main result}\label{ss_aim}

We improve the current best upper bound by $\Theta(n^{1/4})$.

\begin{theorem}\label{mainthm}
There exists a constant $\gamma \geq 0.002$ and a number  $n_0$ such that for every  $n> n_0$
 $$S(n) < n^{1/2}+n^{1/4}(1-\gamma).$$
\end{theorem}

The proof is elementary. 
Since this bound is unlikely to be sharp, we did not optimize the computation.
In Section~\ref{sec_lind}, we recall Lindstr\"{o}m's argument~\cite{L69} and point out how introducing a slack term in a critical inequality  can imply possible improvements downstream.
In Section~\ref{sec:hypergraph} we present a different proof, a generalized version of an argument of Ruzsa~\cite{R93}.
We put his proof into a different framework and we explain, how introducing a slack term in a critical inequality there as well can lead  to a possible improvement.
In Section~\ref{sec_proof} we put the two proof methods together and leverage conditions on these slack terms to get Theorem~\ref{mainthm}.
In fact, we show that a very dense Sidon set must have large discrepancy on some initial segment of $[n]$, for related results see Cilleruelo~\cite{C12}, Erd\H{o}s and Freud~\cite{EF}.

Our second aim is to have a self-contained introduction to Sidon sets, so we describe a construction  matching the lower bound as well.
We also give two more examples for Sidon type problems, weak Sidon sets in Section~\ref{sec_weak} and
 thin Sidon sets in Section~\ref{sec_thin}, where our ideas can be used.
Overall, we do not think that our work deserves $0.2\%$ of Erd\H os' prize money, i.e., \$1,
 but we want to emphasize here the wonderful unity of mathematics by showing the many remarkable connections
 Sidon's problem has not only to Fourier analysis  but to abstract algebra and coding theory as well.

\section{Lindstr\"om's Upper Bound}\label{sec_lind}

In this Section we prove~\eqref{eq142}. Recall that $A$ is a Sidon set if all pairwise sums of its elements are distinct up to reordering.
\begin{theorem}[\cite{L69}]\label{X_lind}
Let $A \subseteq [n]$ be a Sidon set. Then $|A| < n^{1/2}+n^{1/4}+1.$
\end{theorem}
\begin{proof}
Let the elements of $A$ be $a_1 < \ldots < a_k$. Given $i < j$, call $j-i$ the \emph{order} of the difference $a_j-a_i$. We sum all 
differences of order at most $\MMM $, where $\MMM $ will be chosen later to be around $n^{1/4}$. There are
$(k-1) + (k-2) + \ldots + (k- \MMM)=
\MMM (k-\frac{\MMM +1}{2})$ such differences. 
By the Sidon property these differences are all distinct. We get the lower bound
\begin{equation}\label{eqlind1}
\frac{1}{2}\MMM ^2\left(k-\frac{\MMM +1}{2}\right)^2 < 1+2+\ldots +\MMM \left(k-\frac{\MMM +1}{2}\right) \leq \sum_{1\leq i<j\leq k, \, j\le  i+\MMM} (a_j-a_i).
  \end{equation}
On the other hand, due to cancellation, the sum of all differences of order $r$  is at most $r\times n$.  We get a lower bound on the sum of differences of order at most $\ell$:
  $\binom{\MMM+1}{2} n$ for the right hand side of~\eqref{eqlind1}, hence  we obtain
\begin{equation}\label{X_sumdiff}
    \frac{1}{2}\MMM ^2\left(k-\frac{\MMM +1}{2}\right)^2 < \sum_{1\leq i<j\leq k, \, j\le  i+\MMM} (a_j-a_i) < \frac{1}{2}\MMM (\MMM +1)n.
\end{equation}
Rearranging, we get
   $k -\frac{\MMM +1}{2} < \sqrt{n (\MMM+1)/\MMM }$ \ which, using $\sqrt{1+x}\le 1+x/2$,  leads to
   \[  k < \sqrt{n} +  \frac{\sqrt{n}}{2\MMM } + \frac{\MMM }{2}+ \frac{1}{2}.
   \]
Substituting $\MMM := \lfloor n^{1/4} \rfloor$ yields the claimed upper bound.
\end{proof}

\subsection{A Possible Improvement, Slackness}\label{ss_lind}
If the set of differences $\{(a_j-a_i): 1\leq i<j\leq k, \, j\le  i+\MMM\}$ contains large values, then one can improve the upper bound on $k$. Define the non-negative slack term $C=C(A,\MMM )$
as follows
\begin{equation}\label{X_ineqL}
    C = C(A,\MMM ):= \left(\sum_{1\leq i<j\leq k, \, j\le i+\MMM} (a_j-a_i)\right) - \sum_{1\leq i \leq \MMM (k-({\MMM +1})/{2})} i.
    \end{equation}
 We can add $C$, or any lower bound for it, to the left hand side of~\eqref{X_sumdiff} to obtain
$$k-\frac{\MMM +1}{2} < \sqrt{\left(1+\frac{1}{\MMM }-\frac{2C}{\MMM ^2n}\right)n} .$$
Using $\sqrt{1+x} \leq 1+x/2$ we conclude
$$k < \sqrt{n} + \frac{\sqrt{n}}{2\MMM } + \frac{\MMM }{2} - \frac{C}{\MMM ^2\sqrt{n}} + \frac{1}{2}.
  $$
We shall substitute here $\MMM =(1-\alpha)n^{1/4}$  with some $0\leq \alpha <1$, and use it as follows
\begin{equation}\label{X_ineqL2}
    k < n^{1/2}+n^{1/4} - \frac{(2C/n) -n^{1/4}\alpha^2(1-\alpha)}{2(1-\alpha)^2} + \frac{1}{2}.
\end{equation}
This  clearly shows that if we knew $C=\Omega(n^{5/4})$ and $\alpha$ was sufficiently small, then Theorem~\ref{mainthm} would follow.

\section{Same Bound via Set Systems}\label{sec:hypergraph}

\subsection{An inequality from coding theory}
The following theorem is usually attri\-buted to Johnson~\cite{J62},
who used it to get upper bounds for error-correcting codes.
It was rediscovered several times, e.g., Bassalygo~\cite{Ba65}.
In hypergraph language (as in Lov\'asz's exercise book~\cite{LL}) it is a
 statement about the size of the ground set of a set system with bounded intersection sizes.
\begin{theorem}\label{setsthm}
Let $\mathcal{A}$ be a family  of $k$-sets $A_1,\dots,A_\mmm $ such that the intersection $A_i\cap A_j$ of any distinct pair has cardinality at most $\lll$. Then $v:= |\cup A_i| \geq \frac{k^2\mmm }{\lll m + k - \lll}.$
\end{theorem}
\begin{proof}
Denote  $d_x$ the number of sets $A_i$ containing  the vertex $x$.
Given $A_i$, we have $\sum_j |A_i\cap A_j|\leq (m-1)t +k$.  This leads to the following chain of inequalities.
\begin{equation}\label{X_ineqF}
\lll \mmm(\mmm-1)+\mmm k \geq \sum_i \left(\sum_j |A_i\cap A_j|\right) = \sum d_x^2 \geq  \frac{(\sum d_x)^2}{v}= \frac{k^2\mmm^2}{v}.  
\end{equation}  \end{proof} \hfill

\subsection{A second combinatorial proof}\label{XXXXsec:hypergraph}

Here we prove~\eqref{eq143}.
The proof is a ge\-neralized version of an argument of Ruzsa~\cite{R93}, but we are going to use Johnson's inequality directly.
The slight improvement of Lindstr\"om's bound (1/2 instead of 1) is due to Cilleruelo~\cite{C10}
 who recalculated Ruzsa's argument more carefully. 

\begin{theorem}\label{X_thm32}
Let $A \subseteq [n]$ be a Sidon set. Then $|A| < n^{1/2}+n^{1/4}+\frac{1}{2}.$
\end{theorem}

\begin{proof}
Set $k:=|A|$, $\mmm$ a positive integer, and $A_i:=A+(i-1)$ for $i=1,\dots,\mmm$.
Note that $\cup A_i \subseteq [n+\mmm -1]$. The crucial observation is that $|A_i\cap A_j|\leq 1$ for $i\neq j$. Indeed, if there are two distinct elements $x,y \in A_i \cap A_j$ for some $i < j$ then there would exist elements $a,b,c,d \in A$ with $a+i=x=b+j$ and $c+i=y=d+j$, which would give us $a+d=b+c$. Since $A$ is a  Sidon set, this forces $\{a,d\}=\{b,c \}$ hence $x=y$, and Theorem~\ref{setsthm} is applicable with $\lll = 1$,  $v \leq n+\mmm -1$, and we get
$$(n+\mmm-1)(\mmm +k-1) \geq {k^2\mmm }.  $$
Suppose $k\geq n^{1/2}+n^{1/4}$ and define $m:= \lceil n^{3/4}\rceil$. Then $(n+m-1)/k < \sqrt{n}$, so we obtain  $\sqrt{n} (\mmm +k-1) > {km }. $
 This leads to
  $$ k < \sqrt{n} \frac{m-1}{m-\sqrt{n}}\leq n^{1/2} \frac{n^{3/4}-1}{n^{3/4}-n^{1/2}}=  n^{1/2}+n^{1/4}+1.$$
A more careful calculation, see Claim~\ref{le63}, completes the proof of Theorem~\ref{X_thm32}.
\end{proof}

\subsection{A possible improvement using variance}
Consider the family $\mathcal{A}$ in Theorem~\ref{setsthm}.
One can naturally get the idea that the lower bound in~\eqref{X_ineqF} can be improved if one knows that the variance of the degree sequence $\{d_x\}$ is large.
Define the non-negative {\it defect} term $K=K(\mathcal{A})$ as the difference
$$K(\mathcal{A}):=\sum d_x^2 - \frac{(\sum d_x)^2}{v}\quad =\sum_{x} (d_{ave}-d_x)^2.$$
We also call $K(\mathcal{A})$ or any lower bound $K$ of it a ``{\it gain}''.
Instead of~\eqref{X_ineqF} one obtains
\begin{equation}\label{ineqF}
\lll \mmm(\mmm-1)+\mmm k \geq   \frac{k^2\mmm^2}{v} + K.
\end{equation}
 Using  $\lll=1$ and  $v \leq n+\mmm -1$  we get
$$\mmm +k-1 > \frac{k^2\mmm }{n+\mmm } +\frac{K}{\mmm}. $$
This rearranges to
$$\left(k-\frac{n+\mmm }{2\mmm }\right)^2 < (n+\mmm )\left(1+\frac{\mmm+n}{4\mmm^2}-\frac{1}{\mmm}-\frac{K}{\mmm^2}\right)
 < (n+\mmm )\left(1-\frac{4K-n}{4\mmm^2}\right). $$
Define $\mmm:=\lfloor n^{3/4} \rfloor$, and suppose that $K< 2n^{3/2}$, as otherwise one could easily obtain much stronger results.
Using $\sqrt{1+x} \leq 1+x/2$, 
  we have
\begin{equation}\label{X_ineqF2}
    k <  \frac{n+\mmm }{2\mmm } + \sqrt{n}\cdot \sqrt{1+\frac{\mmm}{n} -\frac{4K-n}{4\mmm^2}- \frac{4K-n}{4\mmm n} } \leq  n^{1/2}+n^{1/4}-\frac{K}{2n}+2.
\end{equation}
We shall use later that if 
  $K> 2\gamma n^{5/4}+4n$, then~\eqref{X_ineqF2} implies Theorem~\ref{mainthm}.

\smallskip



\noindent Recall the following corollary of the Cauchy-Schwarz inequality about the variance of numbers.

\begin{lemma}\label{defect}
Let $(y_1, \ldots, y_v)$ be a sequence of real numbers with average $d$.
For a subset $X\subseteq [v]$ the average of the elements of $\{ d_x: x\in X\}$ is denoted by $d_X$. 
Then $\sum(d-y_i)^2 \geq |X| (d-d_X)^2$.
\end{lemma}

\begin{proof} 
We have
$$\sum_{i\in [v]} (d-y_i)^2 \geq
   \sum_{x\in X} (d-y_x)^2= |X| (d-d_X)^2 + \sum_{x\in X} y_x^2 -|X|d_X^2.
$$ \end{proof}

\section{Putting the two methods together:  Proof of Theorem~\ref{mainthm}}\label{sec_proof}

\begin{proof}[Proof of Theorem~\ref{mainthm}]
The proof is a combination of the proofs in the previous two sections.
Recall that $A=\{ a_1, a_2, \dots, a_k\}\subset [n]$ is a Sidon set,   $a_1 < \ldots < a_k$,
 $\mmm=\lfloor n^{3/4} \rfloor$ is a positive integer,
 $\mathcal{A}$ is the family $\{ A_i: A_i:=A+(i-1)$ for $i=1,\dots,\mmm\}$ with degree sequence  $\{ d_1, \dots, d_{n+m-1}\}$.
We may suppose that $n^{1/2}+ \frac{1}{2} n^{1/4}< k < n^{1/2}+ n^{1/4}+1/2$, hence the
average degree $d = \frac{k\mmm}{n+m-1} = n^{1/4}+O(1)$.

In this section we fix a ``small" $\alpha>0$ and a ``smaller" $\beta$, and an $ \eps$ to
 get a positive $\gamma$ satisfying~\eqref{eq41}.
E.g., one can choose $\alpha = 0.137$, $\beta = 0.037$, $\eps=0.235$
  and any $\gamma$ with $0.00204\geq \gamma> 0.002$, then these values satisfy
\begin{equation}\label{eq41}
\min\left\{{\eps^2\beta}  , \frac{2(1-\alpha-2\eps)^2(\alpha - 2\beta)-\alpha^2(1-\alpha)}{2(1-\alpha)^2} \right\} > \gamma .
\end{equation}
We also define $s=\lfloor \beta n^{3/4} \rfloor$, \ $r_1=|A \cap [s]|$,\  $r_2=|A \cap [n+1-s,n]|$, $r=r_1 + r_2$,
$R_1:=|A  \cap  [m-s]|$, $R_2:=|A \cap [n+1-m+s, n]|$, $R=R_1+R_2$, and   $\MMM =\lfloor (1-\alpha)n^{1/4}\rfloor$.
In the course of the proof we distinguish three cases:
$r\leq 2(1-\eps) n^{1/4}$, and $R\geq 2(1+\eps)n^{1/4}$, and $2(1-\eps) n^{1/4}< r\leq R< 2(1+\eps)n^{1/4}$.

\subsection{The density of the initial segments of $A$.}\label{ss41}\quad
The first main idea is to have  a closer look at the variance of the degree sequence of $\mathcal A$.

\medskip
\Claim{4.1}{If $r\leq 2(1-\eps) n^{1/4}$ then $K\geq 2\eps^2 \beta {n^{5/4}}+O(n)$.}\label{cl41}
\medskip

\begin{proof}
Let $d_X = ({\sum_{ x\in X} d_x})/|X|$ be the average of the degrees for a set  $X\subset [n+\mmm -1]$.
By the definition of the defect $K(\mathcal A)$ and by Lemma~\ref{defect} we have $K \geq |X|(d-d_X)^2$.
For $X=[s] \cup [n+m-s,n+m-1]$ we have   
 $$d_X\ =\ \frac{1}{|X|}\sum_{i\in X} d_i \ = \ \frac{1}{2s}\sum_{1\leq j \leq s} \left(|A \cap [j]| + |A \cap [n+1-j,n]|\right) $$ $$\leq \frac{1}{2} |A \cap ([s] \cup [n+1-s,n])|\ =\  \frac{r}{2} \ \leq \ (1-\eps)n^{1/4}.
  $$
Then, using \ $d-\frac{r}{2}\geq \eps{n^{1/4}}+O(1)$, \ $|X|=2s$,\  and \ $s =\beta n^{3/4}+O(1)$, \ we have the required bound.
\end{proof}

\medskip
\Claim{4.2}{If \ $R\geq 2(1+\eps) n^{1/4}$\  then\  $K\geq 2\eps^2 \beta {n^{5/4}}+O(n)$.}\label{cl42}
\medskip

\begin{proof}
Set $X=[\mmm-s+1,\mmm] \cup [n,n+s-1]$.
Every $x \in [\mmm-s+1,\mmm]$ gets covered $|A \cap [\mmm-s]|=R_1$ times just by the translates of $A\cap [\mmm-s]$. Similarly, every $x \in [n,n+s-1]$ gets covered at least $|A \cap [n+1-\mmm+s,n]|=R_2$ times.
Hence, $|d_X-d|\geq \frac{R}{2}-d \ge \eps n^{1/4}$,  and using \ $|X|=2s$,\  and \ $s =\beta n^{3/4}+O(1)$, \ with the  application of  Lemma~\ref{defect} completes the proof.
  \end{proof}


\subsection{Large gaps in $A$}\label{ss42}\quad
The segment $[s+1,\mmm -s]$ contains $R_1-r_1$ members of $A$, similarly $[n+1-\mmm +s,n-s]$ contains $R_2-r_2$
of them, which  adds up to $R-r$.
After having Claims~\ref{ss41} and~\ref{ss42} 
 we may assume that  $2(1-\eps) n^{1/4}< r\leq R< 2(1+\eps)n^{1/4}$,
hence $A\cap \left([s+1,\mmm -s] \cup [n+1-\mmm +s,n-s]\right)$ has fewer than $R-r< 4\eps n^{1/4}$ elements.
Using these assumptions, we shall slightly modify  the proof of Theorem~\ref{X_lind} by defining
$\MMM =\lfloor (1-\alpha)n^{1/4}\rfloor$.

The second idea of the proof is that with  $\MMM$ defined as above we can find many differences $a_j-a_i$ of small order which are significantly larger than $k\MMM$, hence we give a lower bound $C(A, \ell)$, which was defined in~\eqref{X_ineqL}, that provides the right hand side bound in \eqref{eq41}.

\medskip
\Claim{4.3}{If $2(1-\eps) n^{1/4}< r\leq R< 2(1+\eps)n^{1/4}$ and  $\MMM$, $k$, and $\mmm$ are defined as above, then $$
C(A, \ell)>  (1-\alpha-2\eps)^2(\alpha - 2\beta)n^{5/4}+O(n). $$
}\label{cl43}

\begin{proof}
Consider the pairs $(a_i,a_j)$ with $a_i\le s < \mmm-s < a_j$ and $j\leq i+\ell$.
Each such pair appears in the definition of $C(A,\ell)$ and each of such difference $a_j-a_i$ is at least
  $\mmm-2s$ which exceeds $\MMM (k-(\MMM +1)/2)$, as $\alpha > 2\beta$.
 Each such pair adds a gain  at least $\mmm-2s - (\MMM (k-(\MMM +1)/2)) \ge (\alpha-2\beta)n^{3/4}+O(n^{1/2})$ toward $C$ in~\eqref{X_ineqL}.

Given $a_i$ with $1\leq i\leq r_1$ we  choose $j$ as  $R_1< j \leq \ell +i$; there are
  $\ell-R_1+i$ possibilities, whenever this last quantity is positive.
Altogether we have at least $1+2+\ldots  +(\MMM-R_1+r_1)> (\MMM-R_1+r_1)^2/2$ such pairs.

Similarly, pairs $(a_i,a_j)$ with $a_i < n+1-\mmm+s < n+1-s \le a_j$ and $j\leq i+\ell$ gives us more than $(\MMM-R_2+r_2)^2/2$ such pairs.

Since\ $R_1-R_2-r_1-r_2=R-r\le   4\eps n^{1/4}$, the total number of pairs is at least  $(2\MMM-R+r)^2/4> (1-\alpha-2\eps)^2n^{1/4}+O(1)$,
we get the lower bound for the total gain as stated.
\end{proof}

{\it Completing  the proof of Theorem~\ref{mainthm}:\quad}
The constraints in the above three claims cover all possibilities for $A$.
In the cases covered by Claims~\ref{ss41} and~\ref{ss42} we have
 got a large  defect $K(\A)$ so inequalities~\eqref{eq41} and~\eqref{X_ineqF2}
 establish the required upper bound for $k$.
In the case covered by Claim~\ref{ss43} the large slackness term $C$ in the
inequality~\eqref{X_ineqL2} completes the proof of Theorem~\ref{mainthm}.
\end{proof}

\subsection{Remark}\label{ss43}
  Note that  one could optimize somewhat better the parameters $\alpha, \beta$, etc., but we did not see the point computing more digits of the optimal values.
 Much more refining should be possible by exploring the structure of a Sidon set more thoroughly; i.e., giving further conditions on the number of elements in some intervals.  For example, when $a\in A$ is close to $s$,  then one can improve the bounds from Claim~\ref{cl41}, as such $a$ will not contribute much to the degrees  of vertices in $X$. From the other side, when $a\in A$ is close to $1$, then $a$ will participate in larger gaps, and one can improve Claim~\ref{cl43}. The computation is rather delicate, and it improves the bound on $\gamma$ to $0.00342$. It is likely that it could be improved a bit further, with an additional analysis of the structure of $A$.

 There are discussion of the structure of dense Sidon sets in the blog of Gowers~\cite{gowers}, though probably there is no  obvious connection toward our proof.
It seems that further ideas are needed to get rid of the $n^{1/4}$ term in its entirety.

\section{Analogous Improvements for Weak Sidon Sets}\label{sec_weak}
A set of integers $A = \{ a_i \}$ is  a \emph{weak Sidon} set if all the pairwise sums $a_i+a_j$ are different for $i<j$. Note that for Sidon sets we 
  need this to hold for $i \leq j$. Denote by $W(n)$ 
the size of the largest weak Sidon subset of $[n]$. The current best upper bound
\begin{equation}\label{eq5W} 
W(n) \leq n^{1/4}+\sqrt{3}n^{1/4}+O(1)
  \end{equation} 
is due to Kayll~\cite{K04}, who used a method analogous to Lindstr\"{o}m's~\cite{L69}. One could prove  \eqref{eq5W}  with the set systems method described in Section~\ref{sec:hypergraph}, see Ruzsa~\cite{R93}. Similarly,  as for the Sidon sets, exploiting the tension between the two arguments, we  get a slight improvement in the coefficient of the $n^{1/4}$ term.

\begin{theorem}\label{X_weakthm}
There exists a constant $\gamma \geq 0.0089$ such that $W(n) \leq n^{1/2}+n^{1/4}(\sqrt{3}-\gamma) + O(1)$.
\end{theorem}
The proof is almost identical to the proof of Theorem~\ref{mainthm}, therefore we just sketch it, highlighting only the salient differences.

\subsection{Counting differences up to order $\MMM $}
We proceed almost exactly as in Section~\ref{sec_lind} and we use  the same notation. Take a weak Sidon set $1\leq a_1 < \ldots < a_k\leq n$ and sum all positive differences of order at most $\MMM $.
The important difference is that some of these differences may be repeated.
Let
$$\cP:= \left\{ p: 0< p=a_j-a_i = a_{j'}-a_{i'} \text{ for some } \{i,j\}\ne \{i',j'\}\right\}$$ be the set of repeated distances.
If a difference is repeated, then the participating members form an arithmetic progression of length $3$,
 either $a_i=a_{j'}$ or $a_j=a_{i'}$, so each member of $\cP$ is repeated exactly twice.
Additionally, each member of $A$ could be at most once in the middle of such progression, implying
  $|\cP|\leq k-2$.
Hence, in the left hand side of~\eqref{eqlind1} we can have the sum of positive integers only up to $\MMM (k-(\MMM+1)/2)-(k-2)$.
Instead of   (\ref{X_ineqL}) we obtain  the following.

\begin{equation*}
    \frac{1}{2}\MMM ^2(k-(\MMM +1)/2)^2 + C \leq \frac{1}{2}\MMM (\MMM +1)n + \MMM k^2,
\end{equation*}

\noindent where $C$ is any lower bound for
 $\sum \max\{ a_j-a_i -(\MMM-1)k, 0\}$ for $i<j\leq i+\MMM$.
We set $\MMM =(\sqrt{3}-\alpha)n^{1/4}$, a bit less than what Kayll's proof used and obtain

\begin{equation}\label{X_ineqWL2}
    k \leq n^{1/2}+\sqrt{3}n^{1/4}+\frac{n^{1/4}\alpha^2}{2(\sqrt{3}-\alpha)} - \frac{C}{(\sqrt{3}-\alpha)^2n} +O(1).
\end{equation}

\subsection{Double counting with set systems}
We  
proceed as in Section~\ref{sec:hypergraph}, with $A_1=A$ being our weak Sidon set, $A_i=A+(i-1)$ for $i \in [\mmm]$ and $\mmm =\lfloor \sqrt{3}n^{3/4}\rfloor$.
Since  
 there are at most $k-2$ repeated distances, we conclude that the number of pairs of sets with intersection of size $2$ is at most $k\mmm$, and there is no
 intersection of size larger than $2$. 
We get the following version of (\ref{ineqF}).
\begin{equation*}
\mmm(\mmm-1)+2\mmm k \geq \sum_i \left(\sum_j |A_i\cap A_j|\right) = \sum d_x^2 = \frac{1}{v} \left(\sum d_x\right)^2 +K.
\end{equation*}
This turns into the following version of (\ref{X_ineqF2}).

\begin{equation}\label{X_ineqWF2}
    k \leq n^{1/2}+\sqrt{3}n^{1/4}-\frac{K}{2n}+O(1).
\end{equation}

\subsection{Proof of Theorem \ref{X_weakthm}}  The analogues of  Claims~\ref{ss41} and~\ref{ss42}
imply that we can assume that $[s] \cup [n+1-s,n]$ contains more than $2(\sqrt{3}-\eps)n^{1/4}$ elements of $A$ and $[s+1,\mmm-s] \cup [n-m+s+1,n-s]$ contains fewer than $4\eps 
        n^{1/4}$ such elements, otherwise $\gamma> \beta\eps^2$.
Hence,  the number of pairs $1 \leq i < j \leq i+\MMM $ with $a_j - a_i > \mmm-2s$ is at least ${(\sqrt{3}-\alpha-2\eps)^2}\sqrt{n} + O(n^{1/4})$.
This gives us $C \geq  (\sqrt{3}-\alpha-2\eps)^2(\alpha - 2\beta)n^{5/4}-O(n)$ in (\ref{X_ineqWL2}).
We obtain an improvement $\gamma$ in~\eqref{eq5W}
  in the constant in front of $n^{1/4}$ when
$$\min\left\{ {\beta\eps^2}, \frac{(\sqrt{3}-\alpha-2\eps)^2(\alpha - 2\beta)}{(\sqrt{3}-\alpha)^2} - \frac{\alpha^2}{2(\sqrt{3}-\alpha)}  \right\}> \gamma .$$
Choosing $\alpha = 0.273$, $\beta = 0.068$, $\eps=0.363$ makes the left hand side larger than 0.00896, as claimed in  Theorem~\ref{X_weakthm}.
\qed



\section{$\lll$-thin Sidon sets}\label{sec_thin}

A set of integers $A$ is  a $\lll$-\emph{thin Sidon set} if
  $|A \cap (A+c)|\leq \lll$ for every $c\neq 0$,  where $A+B:= \{ a+b: a\in A, B\in B\}$ and $A+c:= A+ \{ c\}$.
In other words, for every $c$, the equation $a_i-a_j=c$ has at most $\lll$ different solutions $a_i,a_j\in A$.
The case $\lll =1$ corresponds exactly to the Sidon sets.
A related notion, the so-called representation function
 $r_A(x) := |\{(a, a')\in A\times  A : x = a + a'\}|$ where $A$ is a subset in a field
 of characteristic zero, is heavily investigated, see, e.g., Cilleruelo~\cite{C12}.
Let $S_\lll(n):=\max\{ |A|: A \subseteq [n],$ $A$ is $\lll$-thin$\}$.
Caicedo, Martos,  and Trujillo proved that
\begin{theorem}[\cite{CMT}]\label{thm61}
We have $S_\lll(n)= (1+o(1))\sqrt{\lll n}$
while $\lll$ is fixed and $n\to \infty$.
\end{theorem}

Using Theorem~\ref{setsthm} we prove in the next subsection the following upper bound.
 \begin{equation}\label{eq62}
 S_\lll (n) < \sqrt{\lll n} + (\lll n)^{1/4} +\frac{1}{2}.
 \end{equation}
The method of Section 4 gives the following improved upper bound; we omit it's proof.
\begin{theorem}\label{thm66}
There exist a constants $\gamma =\gamma_\lll >0$ and $n_t$ such that $S_\lll (n) \leq (\lll n)^{1/2}+(\lll n)^{1/4}(1-\gamma)$ for every
$n> n_{\lll}$.  \hfill \qed
\end{theorem}

\subsection{Upper bound calculation}

\begin{lemma}[Finding an integer]\label{le62}
Suppose that $f(x):= \alpha x^2 +\beta x + \gamma$, where  $\alpha, \gamma>0$, $\beta<0$ are real numbers.
If $\beta^2-4\alpha \gamma -\alpha^2>0$ then there exists a positive integer $\mmm $ with $f(\mmm)<0$.
\end{lemma}

\begin{proof} The equation $f(x)=0$ has two distinct real roots $\mmm _1$ and $\mmm _2$, we may suppose $\mmm _2>\mmm_1$.
These roots are positive since $f(x)>0$ for all $x\leq 0$.
Furthermore,  $f(x)<0$ if and only if $\mmm _1< x< \mmm_2$.
If $\mmm _2-m_1>1$ then we find a positive integer $\mmm $ between $\mmm _1$ and $\mmm _2$.
This condition is equivalent to the constraint Lemma~\ref{le62}.
   \end{proof}

\begin{claim}\label{le63}
Given $n\geq  \lll \geq 1$ define $\kappa:= \sqrt{\lll n} + (\lll n)^{1/4} +\frac{1}{2}$.
Then there is a positive integer $\mmm $ such that
\begin{equation}\label{eq60}
 (n+\mmm -1) (\lll m + \kappa -\lll) < \kappa^2 \mmm.
\end{equation}
\end{claim}

\begin{proof}
Rearranging~\eqref{eq60} one can see that we need a positive integer $\mmm $ satisfying $f(\mmm)<0$ where
  $f(x):= \lll x^2+(n\lll -\kappa^2+\kappa-2\lll)x +(\kappa-\lll)(n-1)$.
We claim that one can apply Lemma~\ref{le62} to this polynomial $f(x)$, where
$\alpha:=\lll$, $\gamma:= (\kappa-\lll)(n-1)$ are both  positive and
\[ \beta= n\lll -\kappa^2+\kappa-2\lll = -2(n\lll)^{3/4} -\sqrt{n\lll} -2\lll +\frac{1}{4} <0.
\]
Finally,  the main constraint in Lemma~\ref{le62} is satisfied since
\begin{multline*}
\frac{1}{2}|\beta| = (n\lll)^{3/4} +\frac{1}{2}\sqrt{n\lll} +\lll -\frac{1}{8}>
 \sqrt{n\lll}\left((n\lll)^{1/4} +\frac{1}{2}\right) \\ =\sqrt{n\lll} \sqrt{k-\frac{1}{4}}
  > \sqrt{  \lll (n-1) (k-\lll) + \frac{1}{4}\lll^2}.
   \end{multline*}
   \end{proof}

\begin{proof}[Proof of~\eqref{eq62}]
Let $A\subset [n]$ be a $k$-element $\lll$-thin Sidon set.
Suppose, on the contrary, that $k\geq \kappa:= \sqrt{\lll n} + (\lll n)^{1/4} +\frac{1}{2}$.  Claim~\ref{le63} implies that
 there exists a positive $\mmm $ such that
\[ n+\mmm -1 < \frac{\kappa^2 \mmm}{\lll \mmm + \kappa -\lll}.
\]
The function $g(x):= x^2\mmm /(\lll m +x-\lll)$ is strictly increasing for $x>0$,
so we obtain that $n+\mmm -1< g(k)$.

Consider the hypergraph $\cA:= \{ A_1, \dots, A_\mmm \}$ with vertex set $[n+\mmm -1]$ where $A_i:=A+(i-1)$.
Theorem~\ref{setsthm} gives $n+\mmm -1\geq g(k)$,   a contradiction.
   \end{proof}

\subsection{A lower bound construction from finite fields}\label{ss62}
The definition of $\lll$-thin Sidon sets naturally extends to every commutative group.
For completeness we recall the construction from~\cite{CMT} that
 in case of $q$ is a power of a prime number and $\lll$ divides $q-1$, one has
 \begin{equation}\label{eq61}
   S_\lll(\ZZ_{(q^2-1)/\lll})\geq q.
 \end{equation}
For these $n=(q^2-1)/\lll$ we have $S_\lll(\ZZ_n)\geq \sqrt{\lll n +1}$.
Since the set of primes  which are $\, \equiv 1 \mod \lll$ are sufficiently ``dense"
 among the integers and trivially
$S_\lll(n)\geq S_\lll(\ZZ_n)$, one concludes the result in Theorem~\ref{thm61}.

To show~\eqref{eq61}  suppose that $\lll$ divides $q-1$.
Let $\FF_q$ be the finite field of size $q$, $g$ a primitive element of $\FF_{q^2}$.
 Define
\begin{equation}\label{eq64}
    A_{q, \lll}:= \{ x\in [(q^2-1)/\lll]:  g^{x+ y(q^2-1)/t} -g  \in \FF_q \textrm{ for some }y\in [t]\}.
\end{equation}
This is a straightforward modification of the Bose--Chowla~\cite{B42,BC} construction (the case $\lll=1$) and one can show that
 $A_{q, \lll}$ is a
$\lll$-thin Sidon set in $\ZZ_{(q^2-1)/\lll}$.
 \hfill \qed

\medskip

{\bf Acknowledgements:}
We thank to Bernard Lidicky assisting the optimization of formulaes.

\newpage
\section{Appendix, More details for the lower bound construction}

Recall a few definitions from introductory abstract algebra.
Let $q$ be a power of a prime, $\FF_q=(\FF_q, +, \cdot)$ the finite field of size $q$, $\FF^*_q= \FF\setminus \{ 0\}$.
For every element $x\in \FF^*$ we have $x^{q-1}=x^0=1$.
There are elements $g$ such that $\FF^* = \{ g, g^2, \dots, g^{q-1}\}$, these are called {\em primitive} elements of $\FF$.
Then $(\FF^*, \cdot)$ is a cyclic group $\ZZ_{q-1}$.

From now on, we suppose that $g$ is a primitive element of $\FF^*_{q^2}$.
The elements of the form $g^{y(q+1)}$, where  $y\in [q-1]$,  corresponds to  $F_q^*$.
The Bose--Chowla Sidon set $A_q\subset [q^2-1]$ is defined as \begin{equation}\label{eq73}
    A_q:= \{ a: 1\leq a\leq q^2-1, \, g^a= g+ f\textrm{ with } f\in \FF_q\}.
\end{equation}
Bose and Chowla~\cite{B42,BC} showed that $A_q$ is a $q$-element Sidon set in $\ZZ_{q^2-1}$, i.e.,  the numbers $\{ a-a': a, a'\in A, a\neq a'\}$ are all distinct $\mod (q^2-1)$. The properties of the set $A_q$ can be found in many places, e.g., in Chapter 27 of the excellent textbook of van Lint and Wilson~\cite{vLW}.

To show~\eqref{eq61}, from now on, suppose that $\lll$ divides $q-1$.
Define
\begin{equation}\label{eq74}
    A_{q, \lll}:= \{ a': 1\leq a'\leq (q^2-1)/\lll, \, \  a' \equiv a \mod((q^2-1)/\lll) \textrm{ for some }a\in A_q\}.
\end{equation}
This is a $\lll$-thin Sidon set in $\ZZ_{(q^2-1)/\lll}$.
The main steps of the proof of this statement  are as follows.
Given $(u,v)\in \FF_q\times \FF_q\setminus \{ (0,0)\}$, one can define the $q$-element set
  $L(u,v):=\{ (x,y): ux+vy=1, \ x,y\in \FF_q\}$.
Each element of $F_{q^2}$ can be uniquely written in the form $x+\theta y$, $x,y\in \FF$, where $\theta$ is the element we join to $\FF_q$ to obtain $\FF_{q^2}$ and defined by an equation $\theta ^2=s$, where $s\in \FF^*_q$ is a quadratic non-residue.
The elements of $L(u,v)$ define the set $E(u,v):= \{ x+\theta y: \ (x,y)\in L(u,v)\}$, a subset of $\FF^*_{q^2}$.
The elements of $E(u,v)$ corresponds to a set $A(u,v):= \{ a: g^a\in E(u,v)\}$. We have $A(u,v)\subset \ZZ_{q^2-1}$.
A crucial fact is that for every  $(u,v)\in \FF_q\times \FF_q\setminus \{ (0,0)\}$ there exists a
  $(u',v')\in \FF_q\times \FF_q\setminus \{ (0,0)\}$ such that $g E(u,v)= E(u',v')$. In other words $A(u,v)+1= A(u',v') \mod (q^2-1)$.

Let $\cL$ (and $\cE$, and $\cA$) be the $q$-uniform hypergraph formed by the $q^2-1$ hyperedges of the form $L(u,v)$
 ($E(u,v)$, $A(u,v)$, respectively). These hypergraphs are $q$-uniform and $q$-regular and they are cyclic, i.e., there is
 an ordering of the vertices of $\cA$, namely the natural cyclic ordering of $\ZZ_{q^2-1}$, such that $x\to x+1$ defines an automorphism.
As Bose notes, his construction is only an affine analog of Singer's~\cite{S38}, the hypergraph $\cL$ can be considered as an incomplete set of lines of an affine geometry.
Each positive integer $x\in \{1, 2, \dots, q^2-1\}$ occurs as a difference $a-a'$, $a,a'\in A_q$ exactly once, unless $x$ is divisible
 by $q+1$. In the latter case $x$ cannot be a difference of the form $a-a'$.

After this preparation, we sketch  the proof of the promised  properties of $A_{q, \lll}$, we  follow~\cite{F96}.
Let $H\subset \FF^*_q$ be a subgroup of the multiplicative group containing $\lll$ elements.
Two pairs $(u_1,v_1),$ $(u_2,v_2) \in\FF_q\times \FF_q\setminus \{ (0,0)\}$  are equivalent
if and only if there exists an $h\in H$ such that $u_1=hu_2$ and $v_1=hv_2$.
Each equivalence class contains exactly $\lll$ pairs.
We write $\langle u,v\rangle $ for the  equivalence class containing the pair $(u,v)$.
Each $\langle u,v\rangle $ defines $t$ elements in $F^*_{q^2}$, $F\langle u,v\rangle:=\{ x+\theta y: (x,y)\in \langle u,v\rangle\} $, and another $t$-element set from $Z_{q^2-1}$, $Z\langle u,v \rangle:=\{ a: g^a\in F\langle u,v\rangle \}$.
The members of $Z\langle u,v \rangle$ have the same residue $\mod (q^2-1)/t$, and among them there is a unique element
 $z\langle u,v \rangle$ such that $1\leq z\leq (q^2-1)/t$.
The line $L(u,v)$ meets exactly $q$ classes, let $A\langle u,v\rangle := \{ z\langle x,y\rangle:  (x,y)\in L(u,v) \}$. We have $A \langle u,v \rangle \subset \ZZ_{(q^2-1)/\lll}$.
It remains to the reader to show that the hypergraph $\cA_{q,\lll}:= \{ A\langle u,v\rangle : \langle u,v\rangle \in \FF_q\times \FF_q\setminus \{ (0,0)\}/H \}$ is $q$-uniform, $q$-regular, cyclic and
each positive integer $x\in \{1, 2, \dots, (q^2-1)/\lll \}$ occurs as a difference $a-a'$, $a,a'\in A_{q,\lll}$ exactly $\lll$ times unless $x$ is divisible by $q+1$. \hfill \qed

\medskip

\noindent{\bf Remark.} In~\cite{F96} it was also observed that $\cA_{q,\lll} $ possesses a polarity relation $\langle u,v\rangle \leftrightarrow \langle x,y\rangle$ if $ux+vy=1$, in other words its incidency matrix could be symmetric. The latest on this see Livinsky~\cite{L21}.

\end{document}